\documentclass[11pt]{article}
\usepackage[utf8]{inputenc}

\usepackage{geometry} 
\usepackage{graphicx} 

\usepackage{tabularx}
\newcolumntype{Z}{>{\centering\let\newline\\\arraybackslash\hspace{0pt}}X}
\usepackage{ulem}
\usepackage{booktabs} 
\usepackage{array} 
\usepackage{paralist} 
\usepackage{verbatim} 
\usepackage{subfig} 
\usepackage{hyperref}
\usepackage{amsmath}
\usepackage{amssymb}
\usepackage{authblk}
\usepackage{stmaryrd}
\usepackage{enumitem}
\setlist{nosep}
\usepackage{bbm}
\usepackage{mathtools}
\usepackage{textcomp}
\usepackage{xcolor}
\usepackage{listings}
\usepackage{float}
\usepackage{centernot}
\usepackage[mathscr]{euscript}
\usepackage{tikz-cd}
\usepackage{rotating}

\newcommand{\Ccal}{\mathcal{C}}

\newcommand{\Setswith}[1]{\PSh(#1)}

\newcommand{\Hom}{\mathrm{Hom}}

\newcommand{\op}{^{\mathrm{op}}}

\newcommand{\too}{\twoheadrightarrow}

\DeclareMathOperator{\im}{im}

\DeclareMathOperator{\PSh}{\mathbf{PSh}}

\DeclareMathOperator{\End}{End}

\tikzset{
  no line/.style={draw=none,
    commutative diagrams/every label/.append style={/tikz/auto=false}},
  from/.style args={#1 to #2}{to path={(#1)--(#2)\tikztonodes}}
	}
\tikzset{symbol/.style={draw=none, every to/.append style={edge node = {node [sloped, allow upside down, auto=false] {$#1$}}}}}

\usepackage{amsthm}
\newtheorem{thm}{Theorem}[section]

\newtheorem{proposition}[thm]{Proposition}
\newtheorem{lemma}[thm]{Lemma}

\theoremstyle{definition}

\newtheorem{definition}[thm]{Definition}

\newtheorem{example}[thm]{Example}
\newtheorem{problem}[thm]{Problem}
\theoremstyle{remark}

\makeatletter
\usepackage{tikz}
\newcommand*\circled[2][1.6]{\tikz[baseline=(char.base)]{
    \node[shape=circle, draw, inner sep=1pt, 
        minimum height={\f@size*#1},] (char) {\vphantom{WAH1g}#2};}}
\makeatother

\usepackage{fancyhdr} 
\usepackage{sectsty}
\allsectionsfont{\sffamily\mdseries\upshape} 

\usepackage[nottoc,notlof,notlot]{tocbibind} 
\usepackage[titles,subfigure]{tocloft} 

\usepackage[nodayofweek]{datetime}
\longdate

\let\theta\vartheta
\let\emph\textit

\date{\vspace{-2em}}

\title{Solution to a problem by FitzGerald}
\author{Jens Hemelaer \thanks{Department of Mathematics, University of Antwerp, 
 Middelheimlaan 1, B-2020 Antwerp (Belgium) \\ email: jens.hemelaer@uantwerpen.be} \\ Morgan Rogers \thanks{Universit\`a degli Studi dell{'}Insubria, Via Valleggio n. 11, 22100 Como CO \\ Marie Sklodowska-Curie fellow of the Istituto Nazionale di Alta Matematica \\ email: mrogers@uninsubria.it}}
 
 \AtBeginDocument{%
    \def\MR#1{}
 }

\begin{document}

\maketitle

\begin{abstract}
FitzGerald identified four conditions (RI), (UR), (RI*) and (UR*) that are necessarily satisfied by an algebra, if its monoid of endomorphisms has commuting idempotents. We show that these conditions are not sufficient, by giving an example of an algebra satisfying the four properties, such that its monoid of endomorphisms does not have commuting idempotents. This settles a problem presented by Fitzgerald at the Conference and Workshop on General Algebra and Its Applications in 2013 and more recently at the workshop NCS 2018. After giving the counterexample, we show that the properties (UR), (RI*) and (UR*) depend only on the monoid of endomorphisms of the algebra, and that the counterexample we gave is in some sense the easiest possible. Finally, we list some categories in which FitzGerald's question has an affirmative answer.
\end{abstract}


\section{Introduction}

In universal algebra, an important invariant of an algebra is its monoid of endomorphisms. There are various research problems concerning these endomorphism monoids. For example, given a variety of algebras, which monoids appear as endomorphism monoids of algebras in this variety? And what properties of an algebra can we deduce from properties of its endomorphism monoid, and vice versa?

We refer to \cite{adams-bulman-fleming} for an overview of known results in this direction. We highlight some results. Each monoid is the endomorphism monoid of some directed graph. This was shown by Hedrl\'in and Pultr, first for finitely generated monoids \cite{hedrlin-pultr}, and then for monoids with cardinality strictly smaller than the first inaccessible cardinal in \cite{hedrlin-pultr-2}, and eventually for all monoids in \cite{hedrlin-pultr-vopenka} together with Vop\v{e}nka. In fact, they show that every small category has a full embedding in the category of directed graphs. Similarly, every variety of algebras has a full embedding in the category of directed graphs, see \cite{hedrlin-pultr-embeddings}. It was shown by Hedrl\'in and Ku\v{c}era that more generally any concrete category can be fully embedded in a variety of algebras, when working in von Neumann--Bernays--G\"odel set theory (including the axiom of global choice), under the additional axiom that there are no measurable cardinals, see \cite{hedrlin-icm}. This motivates the following definition: a category is called \emph{(algebraically) universal} if it contains a full subcategory equivalent to the category of directed graphs. Examples of algebraically universal categories are the category of rings and the category of semigroups, see \cite[Chapter V]{pultr-trnkova}.

At the other end of the spectrum, there are categories of algebras such that the algebra can be uniquely recovered from its endomorphism monoid. This is the case for e.g.\ boolean algebras \cite{schein}, torsion-free abelian groups \cite{puusemp-1} \cite{puusemp-2} and symmetric groups \cite{puusemp-symmetric}.

In \cite{jackson} and later in \cite{leech-pita-costa}, FitzGerald shows that if the monoid of endomorphisms has commuting idempotents, then there are four conditions that are satisfied by the algebra, called (RI), (UR), (RI*) and (UR*). We will recall these properties in Section \ref{sec:background}. The question by FitzGerald was then whether these properties also imply that the monoid of endomorphisms has commuting endomorphisms. In Section \ref{sec:counterexample} we will show that this is not the case, by giving a counterexample. 

For each monoid $S$, there is a canonical algebra with $S$ as its monoid of endomorphisms, namely $S$ itself as a right $S$-set under multiplication. Our counterexample will be of this form. In Section \ref{sec:general}, we show why this is not a coincidence. More precisely, we prove that whenever the properties (RI), (UR), (RI*) and (UR*) hold for an algebra $A$ with endomorphism monoid $S$, then the four properties hold as well for $S$ as a right $S$-set under multiplication. So any counterexample to FitzGerald's problem, reduces in this way to a counterexample involving $S$ as a right $S$-set. Conversely, we show that if $S$ as a right $S$-set satisfies properties (UR), (RI*) and (UR*), then the same holds for all other algebras with $S$ as its monoid of endomorphisms. Finally, we illustrate with an example that if $S$ as a right $S$-set satisfies properties (RI), then it is not necessarily the case that $A$ satisfies (RI) as well.

In Section \ref{sec:affirmative-answer-examples}, we discuss some examples of categories in which FitzGerald's question does have an affirmative answer.

The first named author is a postdoctoral fellow of the Research Foundation -- Flanders (file number 1276521N). The second named author was supported in this work by INdAM and the Marie Sklodowska-Curie Actions as a part of the \textit{INdAM Doctoral Programme in Mathematics and/or Applications Cofunded by Marie Sklodowska-Curie Actions}.

\section{Background and statement of the problem}
\label{sec:background}

Recall that, in universal algebra, a \textbf{signature} or \textbf{language} is a set $\Sigma$ of \textbf{operations}, together with a map $\phi: \Sigma \to \mathbb{N}$ that associates to each operation an \textbf{arity}. The elements of $\phi^{-1}(n)$ are called the \textbf{$n$-ary operations} (they are called resp.\ unary, binary, ternary if $n=1,2,3$). An algebra over a signature $\Sigma$ is a set $A$, together with a function
\begin{equation*}
\mu_A : A^{\phi(\mu)} \to A
\end{equation*}
for each operation $\mu \in \Sigma$.
If $A$ and $B$ are two algebras over the signature $\Sigma$, then a homomorphism $f : A \to B$ is a function such that 
\begin{equation*}
f(\mu_A(a_1,\dots,a_n)) = \mu_B(f(a_1),\dots,f(a_n))
\end{equation*}
for each $n$-ary operation $\mu \in \Sigma$. In particular, if $M$ is a monoid, then we can consider the signature with a unary operation $\mu^{(m)}$ for each element $m \in M$. In this way, each right $M$-set $X$ is an algebra for this theory, with $\mu^{(m)}_X : X \to X$ defined by $\mu^{(m)}_X(x) = x \cdot m$. For two right $M$-sets $X$ and $Y$, the homomorphisms $X \to Y$ as algebras are precisely the morphisms as right $M$-sets.

A \textbf{variety of algebras} is the collection of all algebras over a fixed signature that satisfy some equational laws. For example, if $M$ is a monoid, the variety of right $M$-sets has signature as above, and equational laws
\begin{equation*}
\begin{split}
\mu^{(n)}(\mu^{(m)}(x)) = \mu^{(mn)}(x)
\end{split}
\end{equation*}
indexed by pairs $(m,n) \in M\times M$, as well as $\mu^{(1)}(x) = x$. Note that a homomorphism between algebras in a variety automatically preserves the equational laws, so only the signature matters. In the remainder of the paper, we will use the more convential notations $x\cdot m$ or $xm$, rather than $\mu^{(m)}(x)$.

In \cite{jackson} and more recently in \cite{leech-pita-costa}, D.~FitzGerald suggests the following problem.

\begin{problem} \label{problem:fitzgerald}
Let $A$ be an algebra and $S$ its monoid of endomorphisms. If $S$ has commuting idempotents then $A$ has the properties:
\begin{itemize}
\item[(RI)] The intersection of two retracts of $A$ is also a retract, which is to say that if $A_1 \hookrightarrow A$ and $A_2 \hookrightarrow A$ admit retractions, then so does the inclusion of their pullback $A_3$:
\[\begin{tikzcd}
A_3 \ar[r, hook] \ar[dr, phantom, "\lrcorner", very near start] \ar[d, hook] & A_2 \ar[d, hook]  \\
A_1 \ar[r, hook] & A .
\end{tikzcd}\]
\item[(UR)] To each retract of $A$ corresponds a unique idempotent with that image (equivalently, each retract $R \hookrightarrow A$ has a unique left inverse $A \too R$).
\end{itemize}
and their duals:
\begin{itemize}
\item[(RI*)] The join of two coretracts of A is also a coretract, which is to say that if $A \too A_1$ and $A \too A_2$ admit sections, then so does their pushout $A_3$:
\[\begin{tikzcd}
A \ar[r, two heads] \ar[d, two heads] & A_2 \ar[d, two heads] \\
A_1 \ar[r, two heads] & A_3 \ar[ul, phantom, "\ulcorner", very near start] .
\end{tikzcd}\]
\item[(UR*)] Each coretract $A \too R$ has a unique section (right inverse) $R \hookrightarrow A$.
\end{itemize}
\textbf{If (RI), (UR) and their duals hold, is it true that $S$ has commuting idempotents?}
\end{problem}

In the next section, we will resolve the question by providing a counterexample. In Sections \ref{sec:general} and \ref{sec:affirmative-answer-examples}, we discuss some positive results. We end this section by providing some background regarding the properties (RI) and (RI*).

For an algebra $A$ in a category of algebras $\mathcal{C}$, consider triples $(D,m,n)$ for some object $D$ in $\mathcal{C}$ and $m,n : D \to A$ injective homomorphisms (or equivalently, monomorphisms) in $\mathcal{C}$. We say that two triples $(D,m,n)$ and $(D',m',n')$ are equivalent if there is an isomorphism $u : D' \to D$ such that $mu = m'$ and $nu = n'$. The equivalence class of a triple $(D,m,n)$ is denoted by $[m,n]$. The equivalence classes $[m,n]$ form an inverse semigroup $\mathcal{I}_A$, see \cite[Section 2]{fitzgerald}. The composition $[m,n]\cdot[k,l]$ is given by $[mp,lq]$, where
\begin{equation*}
\begin{tikzcd}
D \ar[r,"{n}"] & A \\
F \ar[r,"{q}"'] \ar[u,"{p}"] & E \ar[u,"{k}"']
\end{tikzcd}
\end{equation*}
is a pullback diagram. If $A$ satisfies (RI), then the elements $[m,n]$ where both $m$ and $n$ have a left inverse form a submonoid of $\mathcal{I}_A$, called the \emph{monoid of partial automorphisms between retracts}. If $A$ satisfies (RI*), then there is dually a \emph{monoid of partial automorphisms between coretracts}. This shows how the properties (RI) and (RI*) arise naturally in the study of partial automorphisms. For more on this, we refer to \cite{fitzgerald}.

\section{Counterexample to the problem}
\label{sec:counterexample}

Take the monoid $S$ defined by generators and relations as follows:
\begin{equation*}
S = \langle{e,f,g : e^2 = e,~f^2=f,~g^2=g,~fg=gf=eg=ge=g, fef=g=efe}\rangle
\end{equation*}
Note that $S$ has 6 elements $1,e,f,g,ef,fe$. In particular, $e$ and $f$ do not commute, and $ef$ and $fe$ are not idempotent.

Consider the category of $S$-sets, viewed as a variety of algebras via the construction described earlier. Let $A$ be the canonical $S$-set, with underlying set $S$ and for each $s \in S$ the $1$-ary operation $\mu^{(s)}$ given by right multiplication by $s$. The monoid of algebra endomorphisms of $A$ can then be identified with the monoid $S$. We shall show that the four properties (RI), (UR), (RI*) and (UR*) hold. 

Consider a retract $r : A \too R$. Then by definition, we can find an homomorphism $j : R \hookrightarrow A$ such that $r \circ j = 1_R$. Now $j \circ r$ is an idempotent of $S$. In particular, the image of $j \circ r$ is either $S$, $eS = \{e,ef,g\}$, $fS=\{f,fe,g\}$ or $gS =\{g\}$. These 4 are pairwise distinct. So for each retract, there is a unique idempotent that has this retract as its image. This shows (UR). Further, we can compute that the intersections of two retracts is again a retract, the only nontrivial case being $eS \cap fS = gS$. This shows (RI).

Now we determine the congruences $\rho$ on $A$ such that the quotient map $A \to A/\rho$ is a retraction (i.e.\ the \textbf{coretracts} in the terminology of \cite{leech-pita-costa}). There are four possibilities:
\begin{itemize}
\item $a~\rho_1~ b$ if and only if $a=b$.
\item $a~\rho_e~ b$ if and only if $a,b \in \{ 1,e \}$ or $a,b \in \{f,ef\}$ or $a,b \in \{ fe,g \}$
\item $a~\rho_f~ b$ if and only if $a,b \in \{ 1,f \}$ or $a,b \in \{e,fe\}$ or $a,b \in \{ef,g\}$
\item $a~\rho_g~ b$ for all $a,b \in A$.
\end{itemize}
So they each correspond to a unique idempotent. This shows (UR*). Further, we have to show that the join of two coretracts is again a coretract. The only nontrivial case is to show that $\rho_e \vee \rho_f = \rho_g$. But note that
\begin{equation*}
g~\rho_f~ef~\rho_e~f~\rho_f~1~\rho_e~e~\rho_f~fe
\end{equation*}
so $\rho_e \vee \rho_f$ identifies all 6 elements of $A$. This shows (RI*). We conclude:

\begin{proposition} \label{prop:counterexample-fitzgerald}
Consider the monoid
\begin{equation*}
S = \langle{e,f,g : e^2 = e,~f^2=f,~g^2=g,~fg=gf=eg=ge=g, fef=g=efe}\rangle.
\end{equation*}
For each $m \in S$, we have a $1$-ary operation on $A = S$ defined by multiplication on the right. Then the monoid of endomorphisms of $A$ (considered as algebra) is $S$. Further, $A$
satisfies properties (RI), (UR), (RI*), (UR*), but $S$ does not have commuting idempotents. This provides a counterexample to Problem \ref{problem:fitzgerald}. 
\end{proposition}

\section{More general algebras}
\label{sec:general}

Now a natural question is whether, for an algebra $A$, the properties (RI), (UR), (RI*) and (UR*) can be formulated purely in terms of the endomorphism monoid $S$ of $A$. We will show that this is the case for the properties (UR), (RI*), (UR*). To do this, we will need the concept of idempotent completions.

\begin{definition} \label{def:idempotent-completion}
Let $\mathcal{C}$ be a small category and let $\mathcal{D} \subseteq \mathcal{C}$ be a full subcategory. Then we say that $\mathcal{C}$ is an \textbf{idempotent completion} of $\mathcal{D}$ if:
\begin{enumerate}
\item for every object $C$ of $\mathcal{C}$ and morphism $e : C \to C$ such that $e\circ e = e$, there is an object $C'$ of $\mathcal{C}$ and morphisms 
\begin{equation*}
\begin{tikzcd}
C' \ar[r,"{i}"',shift right=1] & \ar[l,"{r}"',shift right=1] C
\end{tikzcd}
\end{equation*}
such that $r \circ i = 1_{C'}$ (so $C'$ is a retract of $C$) and $i \circ r = e$;
\item for every object $C$ of $\mathcal{C}$ there is an object $D$ of $\mathcal{D}$ such that $C$ is a retract of $D$.
\end{enumerate}
\end{definition}

Any small category has an idempotent completion, and idempotent completions are unique up to equivalence. Further, equivalent categories have equivalent idempotent completions.

Fix a variety of algebras, and let $\mathcal{C}$ be the category of algebras and algebra homomorphisms in this variety. Take $A$ in $\mathcal{C}$ and $\mathcal{C}_A \subseteq \mathcal{C}$ be the full subcategory consisting of the objects that can be written as a retract of $A$. Further, let $\mathcal{D}_A \subseteq \mathcal{C}_A$ be the full subcategory consisting of only $A$ itself. If we interpret monoids as categories with one object, then
\begin{equation*}
\mathcal{D}_A ~\simeq~ S
\end{equation*}
for $S$ the monoid of endomorphisms of $A$.

\begin{proposition} \label{prop:idempotent-completion}
The category of retracts $\mathcal{C}_A$ is an idempotent completion for $\mathcal{D}_A ~\simeq~ S$.
\end{proposition}
\begin{proof}
The second part of Definition \ref{def:idempotent-completion} holds trivially. For the first part, let $R$ be an object of $\mathcal{C}_A$. Then there are morphisms
\begin{equation*}
\begin{tikzcd}
R \ar[r,"{i}"',shift right=1] & \ar[l,"{r}"',shift right=1] A
\end{tikzcd}
\end{equation*} 
such that $r \circ i = 1_R$. Let $e : R \to R$ be an endomorphism of $R$, with $e \circ e$. Then $e$ factors as
\begin{equation*}
\begin{tikzcd}
R \ar[r,"{\pi}"] & e(R) \ar[r,hook,"{j}"] & R
\end{tikzcd}
\end{equation*}
with $e(R)$ the image of $e$, $\pi(x) = e(x)$ for all $x \in R$, and $j$ the inclusion map. It follows from $e \circ e = e$ that $\pi \circ j = 1_{e(R)}$. Further, by looking at the compositions $\pi \circ r$ and $i \circ j$, we see that $e(R)$ is again in $\mathcal{C}_A$ (i.e.\ a retract of $A$).
\end{proof}

Proposition \ref{prop:idempotent-completion} in particular shows that the idempotent completion (up to equivalence) depends only on the endomorphism monoid $S$ of $A$.

If we interpret $S$ as a category with one object, then there is a functor $\mathcal{A} : S \to \Ccal$ sending the unique object to $A$ and each $s \in S$ to the corresponding endomorphism of $A$. The category $\Ccal$ has all colimits, so from the universal property for categories of presheaves we know that there is an unique colimit-preserving functor $F$ making the following diagram commute.
\begin{equation}\label{eq:F}
\begin{tikzcd}
S \ar[r,"{\mathcal{A}}"] \ar[d,"{\mathbf{y}}"'] & \mathcal{C} \\
\Setswith{S} \ar[ru,dashed,"{F}"']
\end{tikzcd}
\end{equation}
This functor $F$ has a right adjoint
\begin{equation*}
G : \mathcal{C} \longrightarrow \Setswith{S}
\end{equation*}
given by $G(B) \simeq \Hom_{\mathcal{C}}(A,B)$ equipped with right $S$-action such that
\begin{equation*}
(f \cdot s)(a) = f(s(a))
\end{equation*}
for all $f \in \Hom_{\mathcal{C}}(A,B)$, $s \in S$, $a \in A$.

Recall that we write $\mathcal{C}_A$ for the full subcategory of $\mathcal{C}$ consisting of the objects that can be written as a retract of $A$. Let $\check{S}$ be the full subcategory of $\Setswith{S}$ consisting of the objects that can be written as a retract of $S$ (as right $S$-set under multiplication). By Proposition \ref{prop:idempotent-completion}, both $\Ccal_A$ and $\check{S}$ are idempotent completions of $S$, so $\check{S} \simeq \Ccal_A$. The objects of $\check{S}$ are the right $S$-sets of the form $eS$, for $e$ an idempotent of $S$. Because the functor $F$ extends $\mathcal{A}$, it sends $eS$ to the retract $e(A)$ of $A$. This means that $F$ and $G$ are quasi-inverse to each other, when restricted to $\check{S}$ and $\Ccal_A$.

We will now try to use the functors $F$ and $G$ to translate the properties (UR), (UR*), (RI), (RI*) from $\mathbf{PSh}(S)$ to $\Ccal$ or vice versa.

\begin{proposition} \label{prop:bridge-four-properties}
Let $\mathcal{C}$ be a category of algebras. Take $A$ in $\mathcal{C}$ and let $S$ be the monoid of endomorphisms of $A$.
\begin{itemize}
\item If one of the properties (RI), (UR), (RI*), (UR*) holds for $A$, then the same property holds for $S$ (as a right $S$-set).
\item If one of the properties (UR), (RI*), (UR*) holds for $S$ (as a right $S$-set), then the same property holds for $A$.
\end{itemize}
\end{proposition}
\begin{proof}
We first show that (UR) holds for $A$ if and only if it holds for $S$ (as right $S$-set). The property (UR) holds for $A$ if and only if every section $j : R \to A$ has a unique right inverse $h: A \to R$. Because $G : \Ccal_A \to \check{S}$ is an equivalence of categories, sending $A$ to $S$ (as a right $S$-set), this is equivalent to the property that in  $\check{S}$ every section $j : R' \to S$ has a unique right inverse, for arbitrary $R'$. This is in turn equivalent to $S$ (as right $S$-set) satisfying (UR). Analogously, we can show that $A$ satisfies (UR*) if and only if $S$ (as a right $S$-set) satisfies (UR*).

If $S$ (as a right $S$-set) satisfies property (RI*), then this means that for two idempotents $e,f \in S$ there is an idempotent $g \in S$ with $ge = gf = g$ such that
\begin{equation} \label{eq:pushout-in-check-S}
\begin{tikzcd}
S \ar[r,"{e \cdot}"] \ar[d,"{f \cdot}"'] & eS \ar[d,"{g \cdot}"] \\
fS \ar[r,"{g \cdot }"'] & gS
\end{tikzcd}
\end{equation}
is a pushout diagram. Applying $F$ gives
\begin{equation} \label{eq:pushout-in-C}
\begin{tikzcd}
A \ar[r,"{e}"] \ar[d,"{f}"'] & e(A) \ar[d,"{g}"] \\
f(A) \ar[r,"{g}"'] & g(A)
\end{tikzcd}
\end{equation}
and this is again a pushout diagram since $F$ preserves pushouts. So $A$ satisfies (RI*). The proof for the dual statement, i.e.\ if $A$ satisfies (RI) then $S$ satisfies (RI), is analogous.

Finally, we show that if $A$ satisfies (RI*), then $S$ (as a right $S$-set) satisfies (RI*). So suppose that $A$ satisfies (RI*), i.e.\ for two idempotents $e,f \in S$ there is a pushout diagram of the form (\ref{eq:pushout-in-C}) in $\Ccal$ for some idempotent $g \in S$. This is also a pushout in the full subcategory $\Ccal_A$. So after applying $G$, we get a pushout diagram of the form (\ref{eq:pushout-in-check-S}) in $\check{S}$. We claim that this is automatically also a pushout diagram in $\Setswith{S}$. We write down the pushout diagram in $\Setswith{S}$ as
\begin{equation} \label{eq:pushout-in-PSh-M}
\begin{tikzcd}
S \ar[r,"{e \cdot}"] \ar[d,"{f \cdot}"'] & eS \ar[d,"{x \cdot}"] \\
fS \ar[r,"{x \cdot }"'] & xS
\end{tikzcd}
\end{equation}
for some element $x \in S$ (the pushout is cyclic because pushouts of epimorphisms are again epimorphisms). It remains to show that there is an isomorphism $xS \to gS$ sending $x$ to $g$. By applying the universal property of the pushout to the diagram (\ref{eq:pushout-in-check-S}), we find a map $\phi : xS \to gS$ such that $\phi(x) = g$. In order to construct a map $\psi: gS \to xS$ with $\psi(g) = x$ it is enough to show that $xg = x$. Because (\ref{eq:pushout-in-check-S}) is a pushout in $\check{S}$, it determines a pullback diagram in the opposite category $\check{S}\op$:
\begin{equation*}
\begin{tikzcd}
S   & Se \ar[l,"{\cdot e}"'] \\
Sf \ar[u,"{\cdot f}"]   & Sg \ar[l,"{\cdot g}"] \ar[u,"{\cdot g}"'] 
\end{tikzcd}\qquad.
\end{equation*}
The Yoneda embedding preserves limits, so the above is a pullback diagram in $\Setswith{S\op}$, not only in $\check{S}\op$. Further, the morphisms in the diagram are precisely the inclusions, so we find that $Se \cap Sf = Sg$. From the commutativity of (\ref{eq:pushout-in-PSh-M}) it follows that $xe = x = xf$, so $x \in Se \cap Sf = Sg$. So $xg = x$, which means there is a morphism $\psi : gS \to xS$ sending $g$ to $x$. This is an inverse to the map $\phi$ defined above. We conclude that the diagram (\ref{eq:pushout-in-check-S}) is a pushout in $\Setswith{S}$ for arbitrary idempotents $e,f \in S$. In other words, $S$ (as a right $S$-set) satisfies (RI*).
\end{proof}

From the above proposition, we know that if an algebra $A$ is a counterexample to Problem \ref{problem:fitzgerald}, and $S$ is its monoid of endomorphisms, then $S$ (as a right $S$-set) is another counterexample. From this point of view, the counterexample that we gave in Proposition \ref{prop:counterexample-fitzgerald} is the easiest possible counterexample.

In universal algebra, the most interesting varieties of algebras seem to be the ones that combine unary and binary operations (groups, lattices, rings\dots). So a further question is: can we find a counterexample to Problem \ref{problem:fitzgerald} in a variety of algebras with more interesting unary and binary operations? One strategy would be to start with the monoid $S$ from Proposition \ref{prop:counterexample-fitzgerald} and then find an algebra $A$ in the variety with $S$ as monoid of endomorphisms. This is always possible if the category of algebras is universal. Since $S$ satisfies (RI), (UR), (RI*) and (UR*) as right $S$-set, the properties (UR), (RI*) and (UR*) are satisfied by $A$, see Proposition \ref{prop:bridge-four-properties}. So the difficulty is in showing that $A$ satisfies (RI).

\begin{lemma} \label{lmm:A-satisfies-RI}
Let $\mathcal{C}$ be a category of algebras. Take $A$ in $\mathcal{C}$ and let $S$ be the monoid of endomorphisms of $A$. Suppose that for every two idempotents $e,f \in S$ there is a natural number $n$ such that $(ef)^n = (fe)^n$. Then $A$ satisfies (RI).
\end{lemma}
\begin{proof}
If $e,f \in S$ are two idempotents, take a natural number $n$ such that $(ef)^n=(fe)^n$. We will write $h = (ef)^n = (fe)^n$. We have $eh = he = h = fh = hf$. It follows that $h$ is an idempotent. We claim that $h(A) = e(A) \cap f(A)$.

For $a \in A$, $e(h(a)) = h(a) = f(h(a))$, so $h(a) \in e(A) \cap f(A)$. This shows $h(A) \subseteq e(A)\cap f(A)$. Conversely, suppose $x \in e(A) \cap f(A)$. Then $e(x) = x = f(x)$, and it follows that $h(x) = x$. So $x \in h(A)$ and this shows the other inclusion $e(A)\cap f(A) \subseteq h(A)$.
\end{proof}

This leads to a counterexample to Problem \ref{problem:fitzgerald} in each universal category of algebras.

\begin{proposition}
Let $\mathcal{C}$ be a category of algebras that is universal. Then there exists an $A$ in $\mathcal{C}$ that satisfies (RI), (RI*), (UR) and (UR*), and such that its monoid of endomorphism does not have commuting idempotents.
\end{proposition}
\begin{proof}
Take an algebra $A$ in $\mathcal{C}$ with endomorphism monoid $S$, with $S$ the monoid from Proposition \ref{prop:counterexample-fitzgerald}. Such an algebra exists because $\mathcal{C}$ is universal. We know that $S$ as a right $S$-set satisfies (RI), (RI*), (UR) and (UR*). It follows from Proposition \ref{prop:bridge-four-properties} that $A$ satisfies (RI*), (UR) and (UR*). Further, this monoid $S$ satisfies the assumption of Lemma \ref{lmm:A-satisfies-RI}, so it follows that $A$ satisfies (RI) as well. But $S$ does not have commuting idempotents.
\end{proof}

In Proposition \ref{prop:bridge-four-properties}, we showed that the properties (RI*), (UR) and (UR*) depend only on the endomorphism monoid of an algebra. We now show that this is not the case for the property (RI).

\begin{example}
We give an example of an algebra $A$ with endomorphism monoid $S$ such that $S$ satisfies (RI) as right $S$-set, but $A$ does not satisfy (RI). 

Take $S = \langle{e,f,g : e^2=e,~f^2=f,~g^2=g,~ eg=ge=g=gf=fg}\rangle$. The only idempotents in $S$ are $1$, $e$, $f$ and $g$. Other than $1$ and $g$, every element can be written in a unique way as a product of $e$'s and $f$'s, with $e$ and $f$ alternating each other. It follows that $eS \cap fS = \{g\} = gS$ and as a result, $S$ satisfies (RI) as right $S$-set.

Let $A = S \cup \{h,g'\}$. Equip $A$ with the structure of a right $S$-set, where the action on $S$ is the canonical action by multiplication, and further $he = h = hf$, $g'e = g' = g'f$ and $hg = g' = g'g$. Further, we equip $A$ with additional operations:
\begin{itemize}
\item two $0$-ary operations (i.e.\ constants) corresponding to the element $g, g' \in A$;
\item a unary operation $X$ defined as
\begin{equation*}
X(a) = \begin{cases}
h \quad & \text{if } a \notin \{g,g'\} \\
g' \quad & \text{if } a \in \{g,g'\}.
\end{cases}
\end{equation*}
\end{itemize}
We claim that $A$ has endomorphism monoid $S$. We have an inclusion $S \subseteq \End(A)$, defined as follows. The element $e \in S$ acts by multiplication on the left on $S$ and trivially on $h$ and $g'$, and similarly, $f \in S$ acts by multiplication on the left on $S$ and trivially on $h$ and $g'$. The element $g \in S$ sends everything in $S$ to $g$, and sends both $h$ and $g'$ to $g'$. To show that these define endomorphisms of $A$, first note that they preserve the right $S$-action on $A$, and that they fix the two constants $g$ and $g'$. It remains to show that also the unary operation $X$ is preserved. For this, note that $a \in \{g,g'\}$ if and only if $e(a) \in \{g,g'\}$, so $e(X(a)) = X(a) = X(e(a))$. Similarly, $f(X(a)) = X(a) = X(f(a))$. Moreover, $X(g(a)) = g' = g(X(a))$. This shows $S \subseteq \End(A)$.

We now show the other inclusion $\End(A) \subseteq S$. First note that $A$ as an algebra is generated by the element $1 \in A$ (the two elements $1$ and $h$ are generators for the underlying right $S$-set, and further $X(1) = h$). So an endomorphism $\phi$ of $A$ is completely determined by the element $\phi(1)$. Because $g = \phi(g) = \phi(1)g$, we know that $\phi(1) \notin \{h,g'\}$. So $\phi(1) \in S \subseteq A$, but then $\phi \in S$.

Finally, we show that $A$ does not satisfy (RI). Because $S$ has idempotents $1$, $e$, $f$ and $g$, the four retracts of $A$ are $A$ itself, $e(A)$, $f(A)$ and $g(A) = \{g,g'\}$. We compute that $e(A) \cap f(A) = \{ g, h, g' \}$, which does not agree with one of the four possibilities. So $e(A) \cap f(A)$ is not a retract, in particular $A$ does not satisfy (RI).
\end{example}

We expect that other interesting properties of an algebra $A$, other than (RI*), (UR) and (UR*), can be described by only looking at the monoid of endomorphisms of $A$. If a mathematical problem only uses these kind of properties, there are two possible methods of attacking the problem:
\begin{enumerate}
\item We forget about the algebra $A$ and focus on the monoid of endomorphisms $S$.
\item We choose \emph{any} algebra $A'$ with monoid of endomorphisms $S$ and we assume (without loss of generality) that $A = A'$ is this particular algebra.
\end{enumerate}
The second one mirrors the strategy of `building bridges' in the sense of Caramello \cite{caramello-unification}.

\section{Categories for which FitzGerald's problem has an affirmative answer}
\label{sec:affirmative-answer-examples}

We know that there is a counterexample to FitzGerald's problem in any universal category of algebras, like the category of rings, the category of semigroups, or the category of right $\langle{x,y}\rangle$-sets, where $\langle{x,y}\rangle$ is the free monoid on two generators. 

We now list some categories for which the answer to FitzGerald's problem has an affirmative answer.

\begin{example}[The category of sets]
For any set $A$ and element $a \in A$, the singleton $\{a\} \subseteq A$ is a retract. If $A$ has two distinct elements $a,b \in A$, then $\{a\}\cap\{b\} = \varnothing$ is an intersection of retracts that is not a retract itself. So if $A$ satisfies (RI) then $A$ is empty or has a singleton. In each case, the endomorphism monoid is trivial, in particular it has commuting idempotents.
\end{example}

\begin{example}[The category of pointed sets]
Let $A$ be a pointed set. As soon as we have distinct elements $a,b,\ast \subseteq A$, where $\ast$ is the distinguished point, we can define idempotents $e$ and $f$ with
\begin{gather*}
e(a)=e(b)=a,~ e(x) = \ast ~\text{for } x \neq a,b. \\
f(a)=a,~ f(x)=\ast ~\text{for } x \neq a.
\end{gather*}
Since $e$ and $f$ have the same image, $\End(A)$ fails to have (UR). So if (UR) is satisfied, then $A$ has one or two elements. One can check that in these two cases, the endomorphism monoid has commutative idempotents.
\end{example}

\begin{example}[Any abelian category]
Let $\mathcal{C}$ be an abelian category, and take $A$ in $\mathcal{C}$ with endomorphism monoid $S$. Let $e,f \in S$ be two idempotents. There is a decomposition $A = e(A) \oplus (1-e)(A)$. Consider the idempotent $\phi \in S$ that fixes $e(A)$ and sends the elements of $a \in (1-e)(A)$ to $ef(a)$. Then $\im(\phi) = e(A)$, and because of property (UR) this forces $\phi = e$, in particular $\phi$ is zero when restricted to $(1-e)(A)$. So $ef(1-e)(x) = 0$ for all $x \in A$, which shows $ef-efe = 0$. Dually, consider the idempotent $\psi \in S$ that sends $(1-e)(A)$ to $0$ and that sends $a \in e(A)$ to $a + (1-e)f(a)$. Then $\ker(\psi) = (1-e)(A) = \ker(e)$, and because of property (UR*) this forces $\psi = e$. In follows that $(1-e)f(e(a))=0$ for all $a \in A$, which shows $fe - efe = 0$. Since we have both $ef = efe$ and $fe = efe$, we see that $ef = fe$.
\end{example}

\begin{example}[The category of $G$-sets, for $G$ a commutative group] Let $G$ be a commutative group. Let $A$ be a set with a left $G$-action. Let $S$ be the monoid of endomorphisms of $A$. Take an idempotent $e \in S$. For each orbit $A_i$ of $A$, $e(A_i)$ is again an orbit. Assume that $e$ is nontrivial, and take a component $A_i$ such that $e(A_i)\neq A_i$. For an element $g \in G$, define the idempotent morphism $\phi \in S$ as
\begin{equation*}
\phi(x) = \begin{cases}
g(e(x)) \quad & \text{if } x \in A_i \\
e(x) \quad & \text{if }x \notin A_i.
\end{cases}
\end{equation*}
Then $\im(\phi) = \im(e)$, so because of property (UR) we have that $\phi = e$. It follows that $g(e(a)) = e(a)$ for all $a \in A_i$. Since $g \in G$ was arbitrary, this shows that $e(A_i)$ is an orbit consisting of a fixed point $\ast$. If there are distinct fixed points $\ast$ and $\ast'$, then the intersection of the two retracts $\{\ast\}$ and $\{\ast'\}$ would not be a retract, so if $A$ satisfies (RI) then $\ast$ is the unique fixed point. We have now shown that every idempotent $e \in S$ satisfies $e(a) \in \{\ast,a\}$ for all $a \in A$. Now take two idempotents $e,f \in S$. Then
\begin{equation*}
ef(a) = \begin{cases}
a \quad & \text{if } x \in e(A) \cap f(A) \\
\ast \quad & \text{otherwise.}
\end{cases}
\end{equation*}
This shows $ef = fe$. 
\end{example}

\section*{Acknowledgments}

We would like to thank Des FitzGerald for explaining to us the background behind the problem, and for formulating the problem in the first place. 

\bibliographystyle{amsplainarxiv}
\bibliography{../monoidprops}

\end{document}